\documentclass[11pt]{article}

\usepackage[centertags]{amsmath}
\usepackage{amsfonts}
\usepackage{amssymb}
\usepackage{amsthm}
\usepackage{newlfont}
\usepackage{verbatim}

\newtheorem{thm}{Theorem}[section]
\newtheorem*{thm*}{Theorem}
\newtheorem{cor}[thm]{Corollary}

\newtheorem{lem}[thm]{Lemma}

\newtheorem{prop}[thm]{Proposition}
\newtheorem*{prop*}{Proposition}

\newtheorem*{conj*}{Conjecture}

\newtheorem*{dfn*}{Definition}
\theoremstyle{definition}
\newtheorem{rem}[thm]{\textbf{Remark}}
\newtheorem*{rmk*}{Remark}
\newtheorem*{fact*}{Fact}

\theoremstyle{proof}

\newcommand{\snorm}[1]{\Vert#1\Vert}
\newcommand{\abs}[1]{\left\vert#1\right\vert}
\newcommand{\set}[1]{\left\{#1\right\}}
\newcommand{\brac}[1]{\left(#1\right)}

\newcommand{\Real}{\mathbb{R}}

\newcommand{\eps}{\varepsilon}

\newcommand{\K}{\mathcal{K}}
\newcommand{\I}{\mathcal{I}}

\newlength{\defbaselineskip}
\numberwithin{equation}{section}

\begin{document}

\title{Isoperimetric and Concentration Inequalities -
Equivalence under Curvature Lower Bound}
\markright{Equivalence of Isoperimetry and Concentration}
\author{Emanuel Milman\textsuperscript{1}}
\date{}

\footnotetext[1]{School of Mathematics,
Institute for Advanced Study, Einstein Drive, Simonyi Hall, Princeton, NJ 08540, USA.
Email: emilman@math.ias.edu.\\
Supported by NSF under agreement \#DMS-0635607.}

\maketitle

\begin{abstract}
It is well known that isoperimetric inequalities imply in a very general measure-metric-space setting appropriate concentration inequalities. The former bound the boundary measure of sets as a function of their measure, whereas the latter bound the measure of sets separated from sets having half the total measure, as a function of their mutual distance. The reverse implication is in general false. It is shown that under a (possibly negative) lower bound condition on a natural notion of curvature associated to a Riemannian manifold equipped with a density,
completely general concentration inequalities imply back their isoperimetric counterparts, up to dimension \emph{independent} bounds. The results are essentially best possible (up to constants), and significantly extend all previously known results,
which could only deduce dimension dependent bounds, or could not deduce anything stronger than a linear isoperimetric inequality in the restrictive non-negative curvature setting.
As a corollary, all of these previous results are recovered and extended by generalizing an isoperimetric inequality of Bobkov.
Further applications will be described in subsequent works. Contrary to previous attempts in this direction, our method is entirely geometric, continuing the approach set forth by Gromov and adapted to the manifold-with-density setting by Morgan.
\end{abstract}

\pagestyle{myheadings}

\section{Introduction}

Let $(\Omega,d)$ denote a separable metric space, and let $\mu$
denote a Borel probability measure on $(\Omega,d)$.
One way to measure the interplay between the metric $d$ and the measure $\mu$ is by means of an isoperimetric
inequality. Recall that Minkowski's (exterior) boundary measure of a
Borel set $A \subset \Omega$, which we denote here by $\mu^+(A)$, is
defined as $\mu^+(A) := \liminf_{\eps \to 0} \frac{\mu(A^d_{\eps}) -
\mu(A)}{\eps}$, where $A_{\eps}=A^d_{\eps} := \set{x \in \Omega ; \exists y
\in A \;\; d(x,y) < \eps}$ denotes the $\eps$ extension of $A$ with
respect to the metric $d$. The isoperimetric profile $\I =
\I_{(\Omega,d,\mu)}$ is defined as the pointwise maximal function $\I
: [0,1] \rightarrow \Real_+$, so that $\mu^+(A) \geq \I(\mu(A))$, for
all Borel sets $A \subset \Omega$. An isoperimetric inequality measures the relation between the boundary measure and the measure of a set, by providing a lower bound on $\I_{(\Omega,d,\mu)}$ by some function $J: [0,1] \rightarrow \Real_+$ which is not identically $0$. Since $A$ and $\Omega \setminus A$ will typically have the same boundary measure, it will be convenient to also define $\tilde{\I} :[0,1/2] \rightarrow \Real_+$ as $\tilde{\I}(v) := \min(\I(v),\I(1-v))$.

Another way to measure the relation between $d$ and $\mu$ is given by concentration inequalities. The log-concentration profile $\K = \K_{(\Omega,d,\mu)}$ is defined as the pointwise maximal function $\K:\Real_+ \rightarrow \Real$ such that $1 - \mu(A^d_r) \leq \exp(-\K(r))$ for all Borel sets $A \subset \Omega$ with $\mu(A) \geq 1/2$. Note that $\K(r) \geq \log 2$ for all $r \geq 0$. Concentration inequalities measure how tightly the measure $\mu$ is concentrated around sets having measure $1/2$ as a function of the distance $r$ away from these sets, by providing a lower bound on $\K$ by some non-decreasing function $\alpha: \Real_+ \rightarrow \Real_+ \cup \set{+\infty}$, so that $\alpha$ tends to infinity.

The two main differences between isoperimetric and concentration inequalities are that the latter ones only measure the concentration around sets having measure $1/2$, and do not provide any information for small distances $r$ (smaller than $R_\alpha:=\alpha^{-1}(\log 2)$). We refer to \cite{Ledoux-Book} for a wider exposition on these and related topics, and to \cite{MilmanSurveyOnConcentrationOfMeasure1988} for various applications.

\medskip

A well known example is that of the standard Gaussian measure $\mu = \gamma_n$ on $\Omega = \Real^n$, equipped with the standard Euclidean metric $d = \abs{\cdot}$. In this case, it was shown by Sudakov--Tsirelson \cite{SudakovTsirelson} and independently Borell \cite{Borell-GaussianIsoperimetry}, that $\I_{(\Real^n,\abs{\cdot},\gamma_n)} = \I_{(\Real,\abs{\cdot},\gamma_1)} = \phi \circ \Phi^{-1}$, where $\phi(y) = (2\pi)^{-1/2} \exp(-y^2/2)$ and $\Phi(x) = \int_{-\infty}^x \phi(y) dy$, and this easily implies that $\K_{(\Real^n,\abs{\cdot},\gamma_n)}(r) = -\log(1-\Phi(r))$. We remark that it is not hard to verify from this that:
\[
c_1 \leq \frac{\tilde{\I}_{(\Real,\abs{\cdot},\gamma_1)}(v)}{v \sqrt{\log 1/v}} \leq c_2 \;\;\; \forall v \in [0,1/2] ~,
\]
where $c_i>0$ are some constants.

More generally, it is known and easy to see that an isoperimetric inequality implies a concentration inequality, simply by ``integrating'' along the isoperimetric differential inequality. Specifically, if $\gamma : [\log 2,\infty) \rightarrow \Real_+$ is a continuous function, it is an easy exercise to show (e.g. \cite[Proposition 1.7]{EMilmanSodinIsoperimetryForULC}) that:
\begin{equation} \label{eq:i-c}
\begin{array}{c}
\tilde{\I}(v) \geq v \gamma(\log 1/v) \;\;\; \forall v\in[0,1/2] \\
\Downarrow \\
\K(r) \geq \alpha(r) \;\;\; \forall r \geq 0 \;\;\; \textrm{where} \;\; \alpha^{-1}(x) = \int_{\log 2}^x \frac{dy}{\gamma(y)} ~.
\end{array}
\end{equation}
The converse statement, that a concentration inequality implies an isoperimetric inequality, is certainly false in general. This is especially apparent when considering a space $(\Omega,d)$ with bounded diameter $R$, in which case $\K(r) = +\infty$ for all $r > R$; but if the support of $\mu$ is disconnected, we can have $\I(v_i) = 0$ for any finite collection of points $\set{v_i} \subset (0,1)$. Of course, this type of counterexample may also be achieved without demanding that the support of $\mu$ be disconnected, but rather by forcing $\mu$ to have little mass (``necks'') in between massive regions.

\smallskip

We will henceforth assume that $\Omega$ is a smooth complete oriented connected $n$-dimensional ($n\geq 2$) Riemannian manifold
$(M,g)$, that $d$ is the induced geodesic distance, and that $\mu$ is an absolutely continuous measure with respect to the Riemannian volume form $vol_M$ on $M$. The above examples demonstrate that in order to have any chance of showing that concentration inequalities imply isoperimetric ones, we need to impose some further conditions which would prevent the existence of small necks. It is therefore very natural, at least intuitively, to require lower bounds on some appropriate curvatures of $(M,g)$ and $\mu$. In this work, we verify that in that case, concentration indeed implies isoperimetry, with quantitative estimates which \emph{do not depend} on the dimension $n$ of $M$, a feature which is essential both in principle and for applications. All our results equally hold when $n=1$ as well, but this case follows from previously known results (described below) and would require special treatment in our approach, so we have chosen to exclude this case from our setup.

\subsection{Main Results}

\begin{dfn*} \label{def:CA}
We will say that our \emph{smooth $\kappa$-semi-convexity assumptions} are satisfied $(\kappa \geq 0)$ if $d\mu = \exp(-\psi) dvol_M$ where $\psi \in C^2(M)$, and as tensor fields on $M$:
\[
 Ric_g + Hess_g \psi \geq -\kappa g ~.
\]
We will say that our \emph{$\kappa$-semi-convexity assumptions} are satisfied if $\mu$ can be approximated in total-variation by measures $\set{\mu_m}$ so that each $(\Omega,d,\mu_m)$ satisfies our smooth $\kappa$-semi-convexity assumptions. \\
When $\kappa=0$, we will say in either case that our \emph{(smooth) convexity assumptions} are satisfied.
\end{dfn*}

\smallskip

Here $Ric_g$ denotes the Ricci curvature tensor of $(M,g)$ and $Hess_g$ denotes the second covariant derivative. $Ric_g + Hess_g \psi$ is the well-known Bakry--\'Emery curvature tensor, introduced in \cite{BakryEmery} (in the more abstract framework of diffusion generators), which incorporates the curvature from both the geometry of $(M,g)$ and the measure $\mu$. When $\psi$ is sufficiently smooth, our $\kappa$-semi-convexity assumption is then precisely the Curvature-Dimension condition $CD(-\kappa,\infty)$ (see \cite{BakryEmery}).

\begin{thm} \label{thm:1}
Let $\alpha: \Real_+ \rightarrow \Real_+ \cup \set{+\infty}$ denote a non-decreasing function. Then under our convexity assumptions, the concentration inequality:
\begin{equation} \label{eq:c-1}
\K(r) \geq \alpha(r) \;\;\; \forall r \geq R_\alpha ~
\end{equation}
implies the following isoperimetric inequality:
\begin{equation} \label{eq:i-1}
 \tilde{\I}(v) \geq \min(c \; v \gamma(\log 1/v),c_\alpha) \;\;\; \forall v \in [0,1/2] \;\;\; , \;\;\; \textrm{where} \;\; \gamma(x) = \frac{x}{\alpha^{-1}(x)} ~,
\end{equation}
$c>0$ is a universal numeric constant and $c_\alpha>0$ is a constant depending solely on $\alpha$. Moreover, for any choice of parameter $\lambda \in (0,1/2)$, both constants may be chosen to depend solely on $\lambda$ and $\alpha^{-1}(\log 1/\lambda)$.
\end{thm}

\smallskip

We remark that our convention for the inverse of a non-decreasing function will be specified in (\ref{eq:inverse-conv}) of Section \ref{sec:aa}.
Theorem \ref{thm:1} may be generalized to the handle our $\kappa$-semi-convexity assumptions for any $\kappa \geq 0$, by requiring an extra growth condition on $\alpha$ in the case that $\kappa > 0$. These two cases, $\kappa = 0$ and $\kappa > 0$, are conceptually different, in both the required assumptions and the consequent proof complexity, and so we have chosen to present them as two separate theorems, even though Theorem \ref{thm:1} is a particular case of Theorem \ref{thm:2} below.

\smallskip

\begin{thm} \label{thm:2}
Let $\kappa \geq 0$ and let $\alpha: \Real_+ \rightarrow \Real \cup \set{+\infty}$ denote a non-decreasing function so that:
\begin{equation} \label{eq:alpha-cond}
\exists \delta_0 > 1/2 \;\;\; \exists r_0 \geq 0 \;\;\; \forall r \geq r_0 \;\;\; \alpha(r) \geq \delta_0 \kappa r^2 ~.
\end{equation}
Then under our $\kappa$-semi-convexity assumptions, the concentration inequality:
\[
\K(r) \geq \alpha(r) \;\;\; \forall r \geq R_\alpha
\]
implies the following isoperimetric inequality:
\begin{equation} \label{eq:i-2}
\tilde{\I}(v) \geq \min(c_{\delta_0} \; v \gamma(\log 1/v),c_{\kappa,\alpha}) \;\;\; \forall v \in [0,1/2] \;\;\; , \;\;\; \textrm{where} \;\; \gamma(x) = \frac{x}{\alpha^{-1}(x)} ~,
\end{equation}
and $c_{\delta_0},c_{\kappa,\alpha}>0$ are constants depending solely on their arguments. Moreover, if $\kappa > 0$, then the dependence of $c_{\kappa,\alpha}$ on $\alpha$ may be expressed only via $\delta_0$ and $\alpha(r_0)$.
\end{thm}

\smallskip

Since the value of $\alpha(r)$ for $r < R_\alpha$ is irrelevant for both assumption and conclusion in these theorems, one may replace $R_\alpha$ in both theorems by 0; our present formulation emphasizes that it is only the tail behaviour of $\alpha$ which is of importance. We also remark that the constants $c,c_\alpha,c_{\delta_0},c_{\kappa,\alpha}$ above have explicit values and formulas we will provide in Sections \ref{sec:thm1} and \ref{sec:thm2}, and emphasize again that they are dimension \emph{independent}.

\subsection{Optimality}

Up to these constants, Theorem \ref{thm:1} is an almost optimal counterpart to (\ref{eq:i-c}). For instance, under the assumptions of Theorem \ref{thm:1}, we may obtain from (\ref{eq:i-c}) and Theorem \ref{thm:1} that for all $p \geq 1$:
\begin{equation} \label{eq:extra-p}
\tilde{\I}(v) \geq v \log^{1-\frac{1}{p}} \frac{1}{v} \;\; \Rightarrow \;\; \K(r) \geq \brac{\frac{r}{p} + \log^{\frac{1}{p}} 2}^p \;\; \Rightarrow \;\; \tilde{\I}(v) \geq \frac{c}{p} v \log^{1-\frac{1}{p}} \frac{1}{v} ~, \end{equation}
where $c>0$ is a universal constant. Similarly to the Gaussian case, a typical model for these inequalities is obtained by considering $(\Real,\abs{\cdot})$ equipped with the probability measure $\exp(-|x/s_p|^p) dx$, where $s_p>0$ is a scaling factor (our convexity assumptions are indeed satisfied in this case). Note the deterioration in $p$ in the conclusion of (\ref{eq:extra-p}), which is especially apparent in the limit as $p \rightarrow \infty$, where we have under the same assumptions as above:
\begin{equation} \label{eq:extra-loglog}
\tilde{\I}(v) \geq v \log \frac{1}{v} \;\; \Rightarrow \;\; \K(r) \geq (\log 2) \exp(r) \;\; \Rightarrow \;\; \tilde{\I}(v) \geq c v \frac{\log \frac{1}{v}}{\log \log \frac{2}{v}} ~, \end{equation}
where $c>0$ is a universal constant.
It is still not clear whether this phenomenon, which also appeared in our previous joint work with Sasha Sodin \cite{EMilmanSodinIsoperimetryForULC}, is genuine or an artifact of the proof (see Remark \ref{rem:ODE}); however, the expression $x / \alpha^{-1}(x)$ (appearing in (\ref{eq:i-1}),(\ref{eq:i-2})) does appear naturally in other works as well (e.g. \cite{EMilmanSodinIsoperimetryForULC,BartheKolesnikov}).
In any case, it is an easy exercise to show that the extra $\log \log \frac{2}{v}$ factor appearing in (\ref{eq:extra-loglog}) is the worst possible gap one can obtain by this procedure.

\medskip

We also remark that the growth condition (\ref{eq:alpha-cond}) on $\alpha$ in Theorem \ref{thm:2} is necessary when $\kappa > 0$, even in the one-dimensional case. This follows from an example of Chen and Wang \cite{ChenWangOptimalConstantInSubGaussianImpliesLogSob}, improving a previous example of Wang \cite{WangImprovedIntegrabilityForLogSob}. For any $\kappa > 0$ and $0<\delta<1/2$, these authors constructed a measure $\mu = \exp(-\psi(x)) dx$ on $(\Real_+,\abs{\cdot})$ such that $\psi'' \geq -\kappa$ and $\K(r) \geq \delta \kappa r^2 + c_{\delta,\kappa}$ for all $r>0$, and yet $(\Real_+,\abs{\cdot},\mu)$ does not satisfy a log-Sobolev inequality, and hence (in fact, equivalently by \cite{BakryLedoux,LedouxSpectralGapAndGeometry}), it does not satisfy a Gaussian isoperimetric inequality: $\liminf_{v \rightarrow 0} \tilde{\I}(v) / v \sqrt{\log 1/v} = 0$. Moreover, it follows from \cite{WangImprovedIntegrabilityForLogSob} that $1/2$ is an upper bound on the value of $\delta$ in any possible counter example as above (this will be described in more detail below).
This demonstrates that the conclusion of Theorem \ref{thm:2} cannot hold without requiring that (\ref{eq:alpha-cond}) should hold with $\delta_0 \geq 1/2$.

Intuitively, condition (\ref{eq:alpha-cond}) means that we must require that the concentration inequality compensate for the negative curvature $-\kappa$ of the space. Since the curvature tensor is a second-order derivative, a natural condition is then indeed $\alpha(r) \geq \delta_0 \kappa r^2$ with $\delta_0 > 1/2$, in agreement with the above discussion.

\subsection{Previously Known Results}

Several variants of Theorems \ref{thm:1} and \ref{thm:2}, where our concentration assumption is replaced by the weaker assumption that $\int_\Omega \exp(\beta(d(x_0,x))) d\mu(x) < \infty$ for some (any) fixed $x_0 \in \Omega$, were previously considered by various authors, primarily for the case $\beta(r) = \delta r^p$ ($p\geq 1$). Unfortunately, the constants in the conclusion of these previous results always depended on the quantity $\int d(x_0,x) d\mu$, which under suitable normalizations (e.g. $\snorm{\frac{d\mu}{dvol_M}}_{L_\infty} \leq C^n$) may be shown to be at least as large as a \emph{dimension dependent} constant
(see Section \ref{sec:conclude}). Wang \cite{WangIntegrabilityForLogSob,WangImprovedIntegrabilityForLogSob} (see also Bakry--Ledoux--Qian \cite{BakryLedouxQianUnpublished}) showed that the case $p=2$ and $\delta > \kappa/2$ under our $\kappa$-semi-convexity assumptions implies a log-Sobolev inequality, which by the work of Bakry--Ledoux \cite{BakryLedoux} (see also Ledoux \cite{LedouxSpectralGapAndGeometry}) implies the right (Gaussian) isoperimetric inequality. As already mentioned, this result is optimal by a construction of Chen and Wang \cite{ChenWangOptimalConstantInSubGaussianImpliesLogSob}, in the sense that the conclusion is false if $\delta < \kappa/2$. A more elementary approach was proposed by Bobkov \cite{BobkovGaussianIsoLogSobEquivalent}, who considered the cases $p=1,2$ under our convexity assumptions on $(\Real^n,\abs{\cdot},\mu)$. Bobkov's method is in fact very general; his results were generalized to all $p\geq 1$ by Barthe \cite{BartheIntegrabilityImpliesIsoperimetryLikeBobkov}, and some underlying ideas may be carried over (to some extent) to treat general manifolds with density satisfying our convexity assumptions. Barthe and Kolesnikov \cite{BartheKolesnikov} have obtained the most general results in this spirit (see in particular \cite[Theorem 7.2]{BartheKolesnikov}), treating general $p>1$ under the convexity assumptions and $p \geq 2$ under the $\kappa$-semi-convexity ones, again, with dimension dependent constants in the conclusion.

It is a-priori not clear that our results may be compared with these previously known results, since the latter ones deduce a weaker (dimension dependent) conclusion under weaker assumptions. Nevertheless, we will see in Section \ref{sec:applications} that all of these results easily follow from our ones, by deducing the following generalization of Bobkov's isoperimetric inequality from \cite{BobkovGaussianIsoLogSobEquivalent}, originally proved in Euclidean space under our convexity assumptions.
The next corollary generalizes all of these previously known results into a \emph{single} coherent statement, and enables handling arbitrary functions $\beta$. Our version generalizes Bobkov's inequality both by handling the manifold-with-density setting and by addressing the $\kappa$-semi-convexity assumptions. We comment that Bobkov's original proof does not seem to generalize to manifolds (even in view of recent developments in that setting), which explains why Barthe and Kolesnikov had to employ other parallel techniques in \cite{BartheKolesnikov} to handle the manifold case.

\begin{cor} \label{cor:gen-Bobkov}
Given $x_0 \in \Omega$, let $\beta : \Real_+ \rightarrow \Real_+ \cup \set{+\infty}$ denote a non-decreasing function so that:
\begin{equation} \label{eq:beta}
-\log \mu\set{x \in \Omega ; d(x,x_0) \geq r} \geq \beta(r) ~.
\end{equation}
Assume that the space $(\Omega,d,\mu)$ satisfies our $\kappa$-semi-convexity assumptions ($\kappa \geq 0$), and that in addition
$\beta$ satisfies the growth requirement:
\[
\exists \delta_0 > 1/2 \;\;\; \exists r_0 \geq 0 \;\;\; \forall r \geq r_0 \;\;\; \beta(r) \geq \delta_0 \kappa r^2 ~.
\]
Then the following isoperimetric inequality holds:
\[
\tilde{\I}(v) \geq \min\brac{\frac{c_{\delta_0}}{\beta^{-1}(\log 1/v)} v \log 1/v , c_{\kappa,\beta}} \;\;\; \forall v \in [0,1/2] ~,
\]
where $c_{\delta_0},c_{\kappa,\beta} > 0$ depend solely on their arguments.
\end{cor}

The similarity in the formulations of Theorems \ref{thm:1},\ref{thm:2} and Corollary \ref{cor:gen-Bobkov} should not mislead the reader: whereas in the former theorems an optimal $\alpha$ would satisfy $\alpha^{-1}(\log 2) = 0$, $\beta^{-1}(\log 2)$ in the context of the latter corollary will typically be large, and as explained in Section \ref{sec:conclude}, in fact dimension dependent. For instance, in the case of the Gaussian measure $\gamma_n$ on $(\Real^n,\abs{\cdot})$, it is easy to check that $\beta^{-1}(\log 2) \geq c \sqrt{n}$, and so Corollary \ref{cor:gen-Bobkov} yields a bad dimension dependence, whereas the actual isoperimetric inequality satisfied by $(\Real^n,\abs{\cdot},\gamma_n)$ is dimension independent, as recovered by using Theorem \ref{thm:1} with $\alpha(r) \geq c r^2$.

\smallskip

The first dimension \emph{independent} result in this spirit was recently obtained in our previous work \cite{EMilman-RoleOfConvexityCRAS,EMilman-RoleOfConvexity}. It was shown in \cite{EMilman-RoleOfConvexity} that under the convexity assumptions, arbitrarily weak concentration (arbitrary slow $\alpha$ increasing to infinity) implies a linear isoperimetric inequality $\tilde{\I}(v) \geq c_\alpha v$, with $c_\alpha$ depending solely on $\alpha$. Theorem \ref{thm:1} significantly extends our previous result, by showing that whenever $\alpha$ increases faster than linearly (corresponding to stronger-than-exponential concentration), one deduces a better-than-linear isoperimetric inequality, with essentially optimal dependence on $\alpha$. Moreover, we will see in Section \ref{sec:applications} that we can actually also recover (a slightly stronger version of) our main result from \cite{EMilman-RoleOfConvexity}. Theorem \ref{thm:2} extends all these results to the negative curvature setting, which is very useful for applications, and which was well beyond the reach of our previous approach.

Further applications of Theorems \ref{thm:1} and \ref{thm:2},
pertaining to stability and other properties of isoperimetric, functional (e.g. Poincar\'e, log-Sobolev, etc.) and transportation cost inequalities, and to optimal inequalities on compact manifolds-with-density (as a function of the diameter and $\kappa$)
will be deferred to \cite{EMilmanGeometricApproachPartII,EMilmanGeometricApproachPartIII}.
We note that the dimension independence feature was crucial for the applications of \cite{EMilman-RoleOfConvexity}, and is equally cardinal to this work as well.

\medskip

Contrary to our previous approach, which involved diffusion semi-group estimates following Bakry--Ledoux \cite{BakryLedoux} and Ledoux \cite{LedouxSpectralGapAndGeometry},
the proofs of Theorems \ref{thm:1} and \ref{thm:2} are purely \emph{geometric}, based on a generalized version of the Heintze--Karcher comparison theorem due to F. Morgan \cite{MorganManifoldsWithDensity} (see also Bayle \cite{BayleThesis}).
In this sense, we return to the original geometric approach of M. Gromov \cite{GromovGeneralizationOfLevy},\cite[Appendix C]{Gromov}, who generalized P. L\'evy's isoperimetric inequality on the sphere to manifolds with positive Ricci curvature. The L\'evy--Gromov approach for obtaining isoperimetric inequalities on manifolds was subsequently employed in the 80's by several authors, including Buser \cite{BuserReverseCheeger}, Gallot \cite{GallotIsoperimetricInqs} and others.  Recently, Morgan extended the setting to allow for non-Riemannian densities, and obtained a geometric proof of the isoperimetric inequality of Bakry--Ledoux \cite{BakryLedoux} for manifolds-with-density satisfying a $CD(\rho,\infty)$ condition ($\rho > 0$), characterizing in addition the equality cases. In particular, this recovers the Gaussian isoperimetric inequality and its equality cases (the latter were first obtained by Ehrhard \cite{EhrhardGaussianIsopEqualityCases} and Carlen--Kerce \cite{CarlenKerceEqualityInGaussianIsop}).
In contrast to the positively curved case ($\rho>0$), where a comparison to a model space is available (see Section \ref{sec:conclude}), our semi-convexity setting $(\rho = -\kappa \leq 0)$ requires a more delicate analysis, to which end we also fully exploit the first variation of area and its consequences.

\subsection{Organization}

The rest of this work is organized as follows. In Section \ref{sec:prelim} we recall some geometric preliminaries which will lie at the heart of our argument. In Section \ref{sec:1vs2} we provide a first example of how the Heintze--Karcher theorem may be used together with the first variation of area, which will be useful later on. In Sections \ref{sec:thm1} and \ref{sec:thm2} we provide the proofs of Theorems \ref{thm:1} and \ref{thm:2}, respectively, under our additional smoothness assumptions. These assumptions are removed by an approximation argument detailed in Section \ref{sec:aa}. In Section \ref{sec:applications} we show how Theorem \ref{thm:1} and the result of Section \ref{sec:1vs2} may be used to give a direct proof of the Main Theorem from \cite{EMilman-RoleOfConvexity}, and provide a proof of Corollary \ref{cor:gen-Bobkov}. We provide some concluding remarks in Section \ref{sec:conclude}. Weaker versions of Theorems \ref{thm:1} and \ref{thm:2} were announced in \cite{EMilmanGeometricApproachCRAS}.

\medskip

\noindent \textbf{Acknowledgements.} I would like to thank Franck Barthe for his support, interest and comments, sharing his knowledge, and providing many relevant references. I would also like to thank Sasha Sodin, Michel Ledoux and Sergey Bobkov for their interest and comments, and Aryeh Kontorovich who helped me regain interest in this problem. I am also thankful to Jean Bourgain for providing motivation and encouragement. Finally, I thank the anonymous referees for their careful reading and numerous comments, which have helped improve the presentation of this work.

\section{Geometric Preliminaries} \label{sec:prelim}

In this section, we recall some known facts from Geometric Measure Theory and Riemannian Geometry which will play a fundamental role in this work. We refer to \cite{FedererBook,MorganBook4Ed,GiustiBook} for further information on the remarkable results provided by Geometric Measure Theory on the existence and regularity of isoperimetric minimizers, to \cite{GHLBook} for basic background in Riemannian Geometry, and to \cite{BuragoZalgallerBook} for a good combination of both topics. A good expository paper on the results stated in this section is that of F. Morgan \cite{MorganManifoldsWithDensity} (also summarized in the recent \cite[Chapter 18]{MorganBook4Ed}); in fact, our main goal will be to describe some of the details underlying the sketched proof of \cite[Theorem 2, Remark 3]{MorganManifoldsWithDensity}. We refer to \cite{BayleThesis} and \cite[p. 216]{BuserReverseCheeger} for careful formal verifications of some of the statements in \cite{MorganManifoldsWithDensity} and to \cite[Appendix A]{EMilman-RoleOfConvexity} for some further details. To facilitate the exposition, we choose to proceed in a non-formal style, but emphasize that all of the details appear in the cited references above. The entire approach we describe here is due to Gromov \cite{GromovGeneralizationOfLevy}, with further recent generalization due to Morgan \cite{MorganManifoldsWithDensity}.

Throughout this section, we assume that $\Omega$ is a complete oriented $n$-dimensional ($n \geq 2$) smooth Riemannian manifold $(M,g)$, that $d$ is the induced geodesic distance, and that $\mu = \exp(-\psi) dvol_M$ where $\psi$ is a $C^2$ function.

\subsection{Generalized Heintze--Karcher Theorem} \label{subsec:H-K}

The first ingredient we will need is a generalization of the Heintze--Karcher theorem (\cite{HeintzeKarcher},\cite[4.21]{GHLBook}), which is a classical volume comparison theorem in Riemannian Geometry when there is no density present ($\psi=0$). Given a $C^2$ oriented hypersurface $S$ in $(M,g)$, the classical theorem bounds the volume of the one-sided neighborhood of $S$ in terms of the mean-curvature of $S$ and a lower bound on the Ricci curvature of $M$. Recall that the \emph{mean curvature} of $S$ at $x$ is just the trace of the second fundamental form $II_{S,x}$ divided by $n-1$, the dimension of $S$.
We conform to the following \emph{non-standard} convention for specifying the latter's sign: the second fundamental form of the sphere in Euclidean space with respect to the \emph{outer} normal is \emph{positive} definite (formally: $II^\nu_{S,x}(u,v) = g(D_u \nu , v)$ for $u,v \in T_x S$, where $D$ is the covariant derivative). It will be more convenient to work with the trace itself (without dividing by $n-1$), which we will call the \emph{total-curvature} of $S$ at $x$, and denote by $H^\nu_S(x)$.
The following generalization to the case of manifolds-with-density is due to Morgan \cite{MorganManifoldsWithDensity}. Other very precise generalizations were also obtained by Bayle \cite[Appendix E]{BayleThesis}. We first introduce the following:

\begin{dfn*}
The $\mu$-total-curvature of $S$ at $x \in S$ with respect to $\nu$, denoted $H_{S,\mu}^\nu(x)$, is defined as:
\[
H_{S,\mu}^\nu(x) := H_S^\nu(x) - \frac{\partial \psi}{\partial \nu}(x) ~.
\]
\end{dfn*}

\begin{thm}[Generalized Heintze--Karcher (Morgan)] \label{thm:gen-HK}
Let $S$ denote a $C^2$ oriented hypersurface in $(M,g)$, and suppose that for $\kappa \in \Real$:
\[
Ric_g + Hess_g \psi \geq \kappa g ~.
\]
Denote by $V_\mu(r)$ the $\mu$-measure of the region within distance $r>0$ of $S$ on the side of the normal unit vector field $\nu$. Then:
\[
V_\mu(r) \leq \int_S \int_0^{r} \exp(H_{S,\mu}^\nu(x) t - \kappa t^2 /2) dt \; dvol_{S,\mu}(x) ~,
\]
where $dvol_{S,\mu} = \exp(-\psi(x)) dvol_S(x)$.
\end{thm}

\subsection{Existence and Regularity of Isoperimetric Minimizers} \label{subsec:e-r}

The second ingredient we will need is the existence and regularity theory of isoperimetric minimizers on manifolds-with-density, provided by Geometric Measure Theory. An isoperimetric minimizer in $(\Omega,d,\mu)$ of given measure $v \in (0,1)$ is a Borel set $A \subset \Omega$ with $\mu(A) = v$ for which the following infimum is attained:
\[
\mu^+(A) = \inf \set{ \mu^+(B) \;;\; \mu(B) = v } ( \; = \I_{(\Omega,d,\mu)}(v) \; ) ~.
\]
In general, isoperimetric minimizers of given measure need not necessarily exist (consider an absolutely continuous measure $\mu$ on $\Real^2$ whose density is not continuous). Fortunately, in the setup of this section, isoperimetric minimizers of any given measure always exist.

Indeed, given a complete smooth oriented $n$-dimensional Riemannian manifold $(M,g)$ with positive density $\rho \in C^k(M)$ ($k \geq 0$) and \emph{finite} weighted volume $V = \int_M \rho < \infty$, and given $0< v <V$, the local compactness theorem for currents (see \cite[Sections 5.5,9.1]{MorganBook4Ed}, \cite{MorganManifoldsWithDensity}, \cite[Chapter 2]{MorgansStudentThesis}) guarantees the existence of a locally integral current with $\rho$-weighted volume $v$ whose boundary minimizes $\rho$-weighted area. It was shown by Morgan in \cite[Remark 3.10]{MorganRegularityOfMinimizers} that such an area minimizing current must be $C^k$ regular outside a set of Hausdorff dimension $n-8$, extending to the manifold-with-density setting previous regularity results of
Almgren, Bombieri, Federer, Fleming, Gonzalez--Massari--Tamanini, Simons  and others (see \cite[Chapter 8]{MorganBook4Ed} and the references therein); in fact, it was shown by
Bombieri--De Giorgi--Giusti \cite{BombieriDeGiorgiGiusti} that the codimension $8$ above is sharp. It is easy to verify that
$\mu^+(A) = \int_{S} \rho \; dvol_S$ for any Borel set $A \subset M$ with $\mu^+(A) < \infty$ such that $S \subset \partial A$ is sufficiently regular (say $C^2$ smooth) and $\partial A \setminus S$ has Hausdorff codimension strictly greater than $1$ (see \cite[pp. 32-33]{BayleThesis} for more on different ways to define the boundary measure). It therefore follows that the weighted area of the minimizing current's boundary and the Minkowski boundary measure of its support coincide (say for $k \geq 2$). In our setup, the probability measure $\mu$ has weighted volume 1 and $C^2$ density, so we conclude the existence of an isoperimetric minimizer of any given measure (it may be assumed to be an open set), whose boundary is in fact $C^2$ outside a singular set of dimension $n-8$ (compare with \cite[Proposition 3.4.11]{BayleThesis}). The complement of the singular set on the boundary will be called the regular part.

\subsection{First Variation of Area}

Let $S$ denote a $C^2$ oriented hypersurface in $(M,g)$, and let $\Phi_u : S \times (-\eps,\eps) \rightarrow M$ denote a $C^2$ normal variation of compact support and constant velocity $u(x)$ along the unit normal vector field $\nu$ (so $\Phi_u(x,t) = \exp_x(t u(x) \nu(x))$). It is well known (e.g. \cite[5.20]{GHLBook}) when $\mu$ is the Riemannian volume that the first variation of area of $\Phi_u(S,t)$ at $t=0$ is given by integrating the total-curvature: $\delta^1(u) = \int_S H^\nu_S(x) u(x) dvol_S(x)$. It is easy to generalize this to handle more general measures $\mu$ (see \cite[3.4.6]{BayleThesis}, \cite[Proposition 7]{MorganManifoldsWithDensity}), in which case the first variation of the $\mu$-weighted area is controlled by the $\mu$-total-curvature: $\delta^1(u) = \int_S H^\nu_{S,\mu}(x) u(x) dvol_{S,\mu}(x)$. Contrary to other approaches, we will not require the second variation, which involves a far more complicated formula.

It follows immediately by a Lagrange multiplier argument that the regular part $S$ of the boundary of any isoperimetric minimizer $A$ must have constant $\mu$-total-curvature (see \cite[3.4.11]{BayleThesis}); we will denote this constant curvature by $H_\mu(A)$.
By definition, the graph of the isoperimetric profile $\I = \I_{(\Omega,d,\mu)}$  cannot lie above the curve $(-\eps,\eps) \ni t \rightarrow (\mu(A_t),\mu^+(A_t))$, and since they touch at $t=0$, they must be tangent at $(\mu(A), \mu^+(A)) = (v,\I(v))$. The slope of the curve at $t=0$ is just the ratio between the first variations of the boundary-measure and the measure, which is exactly $H_\mu(A)$. Consequently (see e.g. \cite{BavardPansu},\cite{MorganJohnson},\cite[Lemma 3.4.12]{BayleThesis}), we deduce:
\begin{prop}[folklore] \label{prop:H-is-derivative}
Let $A$ denote an isoperimetric minimizer of measure $v \in (0,1)$. Then:
\[\limsup_{\eps\rightarrow 0+} \frac{\I(v+\eps) - \I(v)}{\eps} \leq H_\mu(A) \leq
\liminf_{\eps\rightarrow 0-} \frac{\I(v+\eps) - \I(v)}{\eps} ~.
\]\end{prop}

For all practical purposes, the reader should think of $H_\mu(A)$ as the ``derivative" $\I'(v)$. In fact,
by deriving a second-order differential inequality for $\I$ using the second variation, it is possible to show (\cite{BavardPansu},\cite{MorganJohnson},\cite{BayleThesis},\cite{BayleRosales}) that under our smooth $\kappa$-semi-convexity assumptions, $\I$ is locally semi-concave, implying in particular that the right and left derivatives of $\I$ exist (and hence, coincide with the limits above) for all $v \in (0,1)$, and are equal to each other almost everywhere in $(0,1)$.
Nevertheless, we will refrain from using the second variation or any of its consequences, whose proofs in the non-compact manifold-with-density setting lead to numerous complications and technicalities, and will not require all of this for our approach.

\subsection{Combining Everything} \label{subsec:combining}

It is easy to see that Theorem \ref{thm:gen-HK} does not hold in general when $S$ is a non-smooth hypersurface (even in the presence of a single point singularity), so further justification is required before applying it to $\partial A$, the boundary of an isoperimetric minimizer $A$, which may have singularities. It was first observed by Gromov in \cite{GromovGeneralizationOfLevy} that the structure of these singularities allows such justification.
Indeed, it is known that the regular part of $\partial A$ coincides with the set of all points in $\partial A$ whose tangent cone is contained in a half space (see e.g. \cite{MorganBook4Ed}). This means that $p \in M \setminus \partial A$ will always be reached by a normal ray from a regular point in $\partial A$ (any closest one to $p$). Consequently, Theorem \ref{thm:gen-HK} still applies with $S$ denoting the regular part of $\partial A$ and $V_\mu(r)$ denoting the $\mu$-measure of the set $A_r \setminus A$ (in the case of an outer normal field $\nu$).

\medskip

Combining all the ingredients in this section, we obtain the following theorem, due to Morgan \cite[Theorem 2,Remark 3]{MorganManifoldsWithDensity}, which will serve as our main tool in this work.

\begin{thm}[Morgan] \label{thm:tool}
Assume that for $\kappa \in \Real$:
\[
Ric_g + Hess_g \psi \geq -\kappa g ~.
\]
Let $A \subset M$ denote an isoperimetric minimizer in $(M,g,\mu)$ of given measure $v \in (0,1)$.
Let $S$ denote the regular part of $\partial A$, and let $H_\mu(A)$ denote the constant $\mu$-total-curvature of $S$ with respect to the outer unit normal vector field $\nu$ on $S$. Then for any $r>0$:
\[
\mu(A_r) - \mu(A) \leq \mu^+(A) \int_0^r \exp(H_\mu(A) t + \kappa t^2 / 2) dt ~.
\]
\end{thm}

\begin{rem} \label{rem:closure}
Since $\overline{A_r} = \cap_{\eps > 0} A_{r+\eps}$, the above estimate is also valid when $A_r$ is replaced by its closure $\overline{A_r}$.
\end{rem}

\subsection{Further Remarks}

Under the smoothness assumptions of this section on the density of $\mu$, it is clear that $\mu^+(A) = \mu^+(\Omega \setminus A)$ for a set $A$ whose boundary is sufficiently well-behaved. As explained in Subsection \ref{subsec:e-r}, the regularity of the isoperimetric minimizers is enough to ensure this, and so we conclude that $A$ is a minimizer of measure $v$ if and only if $\Omega \setminus A$ is a minimizer of measure $1-v$. Moreover, it was shown in \cite[Section 6]{EMilman-RoleOfConvexity} that when $\mu$ has density which is only locally bounded from above, then $\I$ is continuous and hence symmetric about $1/2$ (i.e. $\I(v) = \I(1-v)$), even though isoperimetric minimizers may not exist in general, nor satisfy $\mu^+(A) = \mu^+(\Omega \setminus A)$. We remark that it follows from the proof of \cite[Lemma 6.1]{EMilmanRoleOfConvexityInFunctionalInqs} that this condition holds (even after taking limit in the total-variation metric) under our convexity assumptions, and the same proof applies to our semi-convexity ones as well. Consequently, we may always assume throughout this work that $\I$ is continuous and symmetric about $1/2$.

\section{First vs. Second Variation} \label{sec:1vs2}

Let us start by proving the following easy observation, which will be useful later on.

\begin{prop} \label{prop:weak-concave}
Under our smooth convexity assumptions, $\I(v)/v$ is non-increasing on $(0,1)$.
\end{prop}

In fact, a stronger result is known:
\begin{thm}[Bavard--Pansu,Morgan--Johnson,Sternberg--Zumbrun,Bayle,Bayle--Rosales,\\Morgan,Bobkov] \label{thm:concave}
Under our smooth convexity assumptions, $\I(v)$ is concave on $(0,1)$.
\end{thm}
\noindent In dimension $n=1$ (assuming $(M,g) = (\Real,\abs{\cdot})$) this was proved by Bobkov \cite{BobkovExtremalHalfSpaces}, by showing that the isoperimetric minimizers are always given by half-lines $(-\infty,a]$ or $[a,\infty)$. The case $n \geq 2$ was shown by several groups of authors \cite{BavardPansu,MorganJohnson,SternbergZumbrun,BayleThesis,BayleRosales,MorganManifoldsWithDensity} (see also \cite[Appendix A]{EMilman-RoleOfConvexity}), by using the existence and regularity results of isoperimetric minimizers, together with a delicate analysis of the \emph{second} variation of area. Our more elementary approach manages to avoid using the second variation, thereby simplifying the proof of the weaker statement of Proposition \ref{prop:weak-concave}. Since $\I$ is symmetric about $1/2$, Proposition \ref{prop:weak-concave} nevertheless implies the following extremely useful:

\begin{cor} \label{cor:weak-concave}
Under our smooth convexity assumptions, $\inf_{v \in [0,1/2]} \tilde{\I}(v) / v = 2 \I(1/2)$.
\end{cor}

This observation lies at the heart of the argument used to prove the main results of \cite{EMilman-RoleOfConvexity}. We conclude that establishing the concavity of the isoperimetric profile is not essential for that argument, and that the weaker Proposition \ref{prop:weak-concave} can be used instead, simplifying the overall proof.

\begin{proof}[Proof of Proposition \ref{prop:weak-concave}]
Given an isoperimetric minimizer $A$ of measure $v \in (0,1)$, we apply Theorem \ref{thm:tool} (with $\kappa=0$) to $B = \Omega \setminus A$, which is an isoperimetric minimizer of measure $1-v$, and let $r \rightarrow \infty$. Since $\Omega$ is connected, clearly $\mu(B_r) \rightarrow 1$, and so we deduce:
\[
 1 - \mu(B) \leq \mu^+(B) \int_0^\infty \exp(H_\mu(B) t) dt ~.
\]
Since $\mu^+(B) = \mu^+(A)$ and $H_\mu(B) = -H_\mu(A)$, this amounts to:
\[
\frac{\mu(A)}{\mu^+(A)} \leq  \int_0^\infty \exp(-H_\mu(A) t) dt  ~.
\]
If $H_\mu(A) > 0$, this implies that:
\[
H_\mu(A) \leq \frac{\mu^+(A)}{\mu(A)} = \frac{\I(v)}{v} ~;
\]
otherwise the same statement holds trivially. We conclude by Proposition \ref{prop:H-is-derivative} that:
\begin{equation} \label{eq:weak-concave}
 \limsup_{\eps\rightarrow 0+} \frac{\I(v+\eps) - \I(v)}{\eps} \leq H_\mu(A) \leq \frac{\I(v)}{v} ~.
\end{equation}
Using that $\I$ and hence $\I(v)/v$ are continuous on $(0,1)$, it is then immediate to check that this is equivalent to the statement that $\I(v)/v$ is non-increasing, by ``differentiating'' the latter expression.
\end{proof}

\begin{rem} \label{rem:H-positive}
Theorem \ref{thm:concave} and the symmetry of $\I$ actually imply that under our smooth convexity assumptions, $\I$ is non-decreasing on $(0,1/2)$, so when
$\mu(A)$ is in this range it follows by Proposition \ref{prop:H-is-derivative} that $H_\mu(A) \geq 0$; the opposite inequality holds in the range $(1/2,1)$.
\end{rem}

\section{Proof of Theorem \ref{thm:1}} \label{sec:thm1}

We now proceed to the proof of Theorem \ref{thm:1}. Although it is a particular case of Theorem \ref{thm:2}, we prefer to
provide a direct proof of the former, since the proof of the latter is more involved, and since we shall benefit from several observations developed
at this stage later on.

Recall that our convexity assumptions ensure that $\kappa=0$ in our Curvature-Dimension condition. We proceed under our additional smoothness assumptions - the general case will follow by an approximation argument provided in Section \ref{sec:aa}. Our concentration assumption $\K(r) \geq \alpha(r)$, or:
\[
\mu(B) \geq 1/2 \;\; \Rightarrow \;\; 1 - \mu(B_r) \leq \exp(-\alpha(r)) ~,
\]
is easily seen to be equivalent to:
\begin{equation} \label{eq:conc-reversed}
\mu(A) > \exp(-\alpha(r)) \;\; \Rightarrow \;\; \mu(A_r) > 1/2 ~.
\end{equation}

Given an isoperimetric minimizer $A$ of measure $v \in (0,1/2)$ (hence $\I(v) = \mu^+(A)$),
we will denote:
\[
r_{\alpha,v} := \alpha^{-1}(\log(1/v)) ~.
\]
By approximating $\alpha$ if necessary, we may assume that it is strictly increasing and continuous - we will see that the general case follows from this one in Lemma \ref{lem:alpha-approx}. This ensures that $\alpha^{-1}$ is well defined, strictly increasing and continuous, and that $\alpha \circ \alpha^{-1} = \alpha^{-1} \circ \alpha = id$. It is then easy to check that (\ref{eq:conc-reversed}) implies that $\mu(\overline{A_{r_{\alpha,v}}}) \geq 1/2$.
Applying Theorem \ref{thm:tool} and Remark \ref{rem:closure} to $A$ and $r_{\alpha,v}$, we deduce:
\begin{equation} \label{eq:potential-ODE}
1/2 - v \leq \mu(\overline{A_{r_{\alpha,v}}}) - \mu(A) \leq
\mu^+(A) \int_0^{r_{\alpha,v}} \exp(H_\mu(A) t) dt ~.
\end{equation}
Since $H_\mu(A) \leq \I(v)/v$ by
(\ref{eq:weak-concave}), we conclude that:
\[
1/2 - v \leq \I(v) \; r_{\alpha,v} \exp\brac{\I(v)\frac{r_{\alpha,v}}{v}} ~.
\]
Denoting $f(v) := \I(v) \frac{r_{\alpha,v}}{v}$, we obtain the following inequality:
\begin{equation} \label{eq:f-inq}
f(v) + \log f(v) \geq \log\brac{\frac{1}{2v}-1} ~.
\end{equation}
Since $y + \log y$ is increasing on $\Real_+$, we obtain that $f(v) \geq b(v)$, where $b(v)$ denotes
the (unique) solution to the equality case in (\ref{eq:f-inq}).
Next, given $\lambda \in (0,1/2)$, denote:
\[
c_\lambda := \inf_{\delta \in (0,\lambda]} \frac{b(\delta)}{\log 1/\delta} ~.
\]
It is elementary to check that $c_\lambda \geq C (1/2-\lambda) > 0$, for
some numeric constant  $C > 0$. 
We obtain that:
\[
f(v) \geq b(v) \geq c_\lambda \log 1/v \;\;\; \forall v \in (0,\lambda] ~,
\]
or equivalently, recalling the definition of $f$ and the symmetry of $\I$:
\[
\tilde{\I}(v) \geq c_\lambda v \gamma(\log 1/v) \;\;\; \forall v \in (0,\lambda] ~,
\]
where $\gamma(x) = x / \alpha^{-1}(x)$.

To conclude the proof, we may proceed in one of two (equally legitimate) ways. The first employs the concavity of the isoperimetric profile $\I$ under our convexity assumptions, which implies as in Remark \ref{rem:H-positive} that $\tilde{\I}$ is non-decreasing on $[0,1/2]$, yielding the desired conclusion:
\begin{equation} \label{eq:thm1-conc-strong}
\tilde{\I}(v) \geq \sup_{\lambda \in (0,1/2)} c_\lambda \min(v \gamma(\log 1/v), \lambda \gamma(\log 1/\lambda))
\;\;\; \forall v \in (0,1/2] ~.
\end{equation}
The second avoids employing the concavity of the profile, at the expense of an additional constant in the conclusion. This may be preferable to the first way, since it avoids using the second variation of area, and is slightly more in the spirit of the proof of Theorem \ref{thm:2} for the case $\kappa > 0$, described next, where we do not have the concavity available. Indeed, using only that
$\I(v) /v$ is non-increasing on $(0,1)$ by Proposition \ref{prop:weak-concave}, and the symmetry of $\I$, it follows that for any $v \in [\lambda , 1/2]$:
\[
\frac{\I(v)}{\lambda} \geq \frac{\I(v)}{v} \geq \frac{\I(1-\lambda)}{1-\lambda} = \frac{\I(\lambda)}{1-\lambda} ~,
\]
and hence:
\begin{equation} \label{eq:thm1-conc-weak}
\tilde{\I}(v) \geq \sup_{\lambda \in (0,1/2)} c_\lambda \min\brac{v \gamma(\log 1/v), \frac{\lambda}{1-\lambda} \lambda \gamma(\log 1/\lambda)}
\;\;\; \forall v \in (0,1/2] ~.
\end{equation}
Either way completes the proof.

\begin{rem} \label{rem:ODE}
One may wonder whether a better result could be obtained if, instead of using $H_\mu(A) \leq \I(v)/v$ in (\ref{eq:potential-ODE}), we would use that $H_\mu(A)$ coincides with $\I'(v)$ a.e., thereby obtaining an ordinary differential inequality which one may try to solve. In particular, it is natural to ask whether this would resolve the small discrepancy appearing in (\ref{eq:extra-loglog}) for the case $\I(v) = v \log 1/v$. It seems that the answer is negative, since in that case $\I'(v) = \log 1/v - 1$, which up to constants (when $v$ is small) coincides with $\I(v) / v$.
\end{rem}

\section{Proof of Theorem \ref{thm:2}} \label{sec:thm2}

We can now turn to the proof of Theorem \ref{thm:2}. Under our $\kappa$-convexity-assumptions, neither Theorem \ref{thm:concave} nor Proposition \ref{prop:weak-concave} are valid (when $\kappa > 0$), so we will need to compensate by applying the Heintze--Karcher theorem on \emph{both} sides of our isoperimetric minimizers' boundary, rendering the proof slightly more involved. As usual, we will add our smoothness assumptions and assume that $\alpha$ is strictly increasing and continuous; the general case follows by an approximation argument described in Section \ref{sec:aa}. Consequently, $\alpha^{-1}$ is well defined, strictly increasing and continuous, and $\alpha \circ \alpha^{-1} = \alpha^{-1} \circ \alpha = id$.

Let $A$ denote an isoperimetric minimizer of measure $v \in (0,1/2)$.
We will first obtain the conclusion of Theorem \ref{thm:2} in the range $v \in (0,\lambda_0)$, for some constant $\lambda_0 \in (0,1/2)$, and then extend this to $v \in [\lambda_0,1/2)$ by a different argument.

\subsection{Small Sets} \label{subsec:small-sets}

As in the previous section, our concentration assumption $\K(r) \geq \alpha(r)$ implies that $\mu(\overline{A_{r_{\alpha,v}}}) \geq 1/2$, where $r_{\alpha,v} := \alpha^{-1}(\log 1/v)$. To get the optimal bound on $\delta_0$ in our growth condition (\ref{eq:alpha-cond}), we will also define:
\[
r_+ := \inf\set{ r \geq 0; \mu(\overline{A_r}) \geq 1/2} ~.
\]
Clearly $\mu(\overline{A_{r_+}}) \geq 1/2$ and $0<r_+ \leq r_{\alpha,v}$. Similarly, using the concentration assumption for the set $\Omega \setminus \overline{A_{r_+ - \eps}}$ ($\eps>0$) whose $\mu$-measure is greater than $1/2$, and taking the limit as $\eps \rightarrow 0$, we deduce that $1-\mu((\Omega \setminus A)_{r_{-}}) \leq v/2$ for:
\[
r_- := r_{\alpha,v/2} - r_+ ~.
\]
Applying Theorem \ref{thm:tool} and Remark \ref{rem:closure} to $A$ and $r_+$, and to $\Omega \setminus A$ and $r_{-}$, we obtain:
\[  1/2 - v \leq \mu(\overline{A_{r+}}) - \mu(A) \leq \I(v) \int_0^{r_+} \exp(H_\mu(A) t + \frac{\kappa}{2} t^2) dt ~,
\] \[  v/2 \leq \mu((\Omega \setminus A)_{r-}) - \mu(\Omega \setminus A) \leq \I(v) \int_0^{r_-} \exp(-H_\mu(A) t + \frac{\kappa}{2} t^2) dt ~.
\] Denoting $S_{\pm} := \pm H_\mu(A) + \frac{\kappa}{2} r_{\pm}$, we note that $S_{\pm} > 0$ iff $r_{\pm}$ is to the right of both roots of the parabola $\pm H_\mu(A) t + \frac{\kappa}{2} t^2$. Consequently,
we obtain the following rough estimates:
\begin{equation} \label{eq:I-2}
 \frac{1/2 - v}{\I(v)} \leq \begin{cases} r_+ \exp( S_+ r_+) & S_+ \geq 0 \\ \min(r_+,-4/H_\mu(A)) & S_+ < 0 \end{cases} ~,
\end{equation}
\begin{equation} \label{eq:II-2}
 \frac{v}{2\I(v)} \leq \begin{cases} r_- \exp( S_- r_-) & S_{-} \geq 0 \\ \min(r_-,4/H_\mu(A)) & S_- < 0 \end{cases} ~,
\end{equation}
where the bound of $\mp 4 / H_\mu(A)$ above follows easily by bounding the parabola
by the two secant lines from its roots to its vertex.
We will need to handle three different cases, assuming that $v \in (0,\lambda)$:
\begin{enumerate}
\item \textbf{$S_+ < 0$}. It is immediate to check that $\frac{1/2 - v}{v \log 1/v}$ is decreasing on $(0,1/2)$. Consequently, if $\lambda \in (0,1/2)$, the definition of $\gamma$ and (\ref{eq:I-2}) immediately imply:
\[
 \tilde{\I}(v) \geq \frac{1/2 - v}{r_+} \geq \frac{1/2 - \lambda}{\lambda \log 1/\lambda} \frac{v \log 1/v}{r_{\alpha,v}} = \frac{1/2 - \lambda}{\lambda \log 1/\lambda} v \gamma(\log 1/v)  ~.\]
\item \textbf{$S_+ \geq 0$ and $S_- < 0$}. Observe that $H_\mu(A) \leq 8 \I(v) / v$ by (\ref{eq:II-2}), and so (\ref{eq:I-2}) implies:
\[
 \frac{1}{2v} - 1 \leq \I(v) \frac{r_+}{v} \exp(S_+ r_+) \leq \I(v) \frac{r_+}{v} \exp\brac{ 8 \I(v) \frac{r_+}{v} + \frac{\kappa}{2} r_+^2} ~.
\]
Denoting $f(v) = \I(v) \frac{r_+}{v}$, we obtain:
\[
 8 f(v) + \log f(v) + \frac{\kappa}{2} r_+^2 \geq \log\brac{\frac{1}{2v} - 1} \geq c_\lambda \log 1/v ~, \]
where $c_\lambda := \frac{\log\brac{\frac{1}{2\lambda} - 1}}{\log 1/\lambda} < 1$.
We now use our growth condition (\ref{eq:alpha-cond}) on $\alpha$ and the definition of $r_{\alpha,v}$, which guarantee that when $v \leq \exp(-\alpha(r_0))$, we have:
\[
\delta_0 \kappa r_+^2 \leq \delta_0 \kappa r_{\alpha,v}^2 \leq \alpha(r_{\alpha,v}) = \log 1/v ~.
\]
Choosing $0<\lambda \leq \exp(-\alpha(r_0))$ small enough so that $c_\lambda > \frac{1}{2\delta_0}$, which is always possible whenever $\delta_0 > 1/2$, we conclude that:
\[
 8 f(v) + \log f(v) \geq \brac{c_\lambda - \frac{1}{2\delta_0}} \log 1/v ~. \]
As in the proof of Theorem \ref{thm:1}, it is elementary to show that this implies:
\[
 f(v) \geq \frac{e}{8e+1} \brac{c_\lambda - \frac{1}{2\delta_0}} \log 1/v ~. \]
Recalling the various definitions used, this amounts to:
\[
 \tilde{\I}(v) \geq \frac{e}{8e+1} \brac{\frac{\log\brac{\frac{1}{2\lambda} - 1}}{\log 1/\lambda} - \frac{1}{2\delta_0}} v \gamma(\log 1/v) ~. \]
\item \textbf{$S_+ \geq 0$ and $S_- \geq 0$}. Since $S_- \geq 0$, we know that $H_\mu(A) \leq \frac{\kappa}{2} r_-$, and therefore:
\begin{equation} \label{eq:case3-1}
S_+ \leq \frac{\kappa}{2} (r_- + r_+) \leq \frac{\kappa}{2} r_{\alpha,v/2} ~.
\end{equation}
As in case 2, if $0 < v < \lambda \leq 2 \exp(-\alpha(r_0))$ then our growth condition (\ref{eq:alpha-cond}) on $\alpha$ ensures that:
\[
 \delta_0 \kappa r_{\alpha,v/2}^2 \leq \log 2/v ~,
\]
and we deduce from (\ref{eq:I-2}) that:
\[
 \frac{1/2-v}{\I(v)} \leq r_+ \exp(S_+ r_+) \leq r_{\alpha,v} \exp(\frac{\kappa}{2} r_{\alpha,v/2}^2) \leq r_{\alpha,v} \brac{\frac{2}{v}}^{\frac{1}{2 \delta_0}} ~.
\]
Since we assumed that $\delta_0 > 1/2$, it is easy to verify from this that:
\[
\tilde{\I}(v) \geq \frac{1/2 - \lambda}{2^{\frac{1}{2 \delta_0}}} \frac{v^{\frac{1}{2 \delta_0}}}{r_{\alpha,v}} \geq \frac{1/2 - \lambda}{2} e\brac{1-\frac{1}{2\delta_0}} v \gamma(\log 1/v) ~. \]
\end{enumerate}

Summarizing all three cases, we have:
\begin{equation} \label{eq:small}
 \tilde{\I}(v) \geq c_{\delta_0,\lambda_0} v \gamma(\log 1/v) \;\;\; \forall v \in (0,\lambda_0) ~,
\end{equation}
with:
\begin{equation} \label{eq:lambda_0}
\lambda_0 = \lambda_0(\delta_0,\alpha(r_0)) := \min\brac{\exp(-\alpha(r_0)),\exp\brac{-\frac{\log 16}{1 - \frac{1}{2\delta_0}}}} ~,
\end{equation}
and:
\begin{equation} \label{eq:c_0}
c_{\delta_0} := \frac{e}{16e+2} \brac{1 - \frac{1}{2\delta_0}} ~.
\end{equation}

\subsection{Large Sets} \label{subsec:large-sets}

We will now extend this to the range $v \in [\lambda_0,1/2)$. Denote:
\[
r_{e,v} := r_{\alpha,v} + r_{\alpha,1/4} ~.
\]
Since $\mu(\overline{A_{r_{\alpha,v}}}) \geq 1/2$, we can apply our concentration assumption again to obtain $\mu(A_{r_{e,v}}) \geq 3/4$. As usual, Theorem \ref{thm:tool} implies that:
\[
 3/4 - v \leq \mu(A_{r_{e,v}}) - \mu(A) \leq \I(v) \int_0^{r_{e,v}} \exp(H_\mu(A) t + \frac{\kappa}{2} t^2) dt ~.
\]
Denoting $S_{e,v} := H_\mu(A) + \frac{\kappa}{2} r_{e,v}$, we obtain that:
\begin{equation} \label{eq:III}
 \frac{3/4 - v}{\I(v)} \leq r_{e,v} \max(\exp( S_{e,v} r_{e,v}),1) ~.
\end{equation}
When $v \in [\lambda_0,1/2)$, clearly $r_{e,v}$ is bounded above by:
\begin{equation} \label{eq:R_0}
R_0 = R_0(\lambda_0,\alpha) := \alpha^{-1}(\log 1/\lambda_0) + \alpha^{-1}(\log 4) ~,
\end{equation}
so we deduce  from (\ref{eq:III}) that if $S_{e,v} < 0$ then:
\begin{equation} \label{eq:III-1}
\I(v) \geq \frac{1}{4R_0} ~.
\end{equation}
On the other hand, if $S_{e,v} \geq 0$, (\ref{eq:III}) implies:
\begin{equation} \label{eq:III-2}
\I(v) \geq \frac{\exp(-S_{e,v} R_0)}{4R_0}  \geq \frac{\exp(-\frac{\kappa}{2} R_0^2) }{4R_0} \exp(-H_\mu(A) R_0) ~.
\end{equation}
Recalling Proposition \ref{prop:H-is-derivative} stating that $H_\mu(A)$ is essentially the ``derivative'' $\I'(v)$, and
since $\I$ is continuous, it is then easy to realize that the validity of (\ref{eq:III-1}) and (\ref{eq:III-2}) for all $v \in [\lambda_0,1/2)$ implies:
\begin{equation} \label{eq:large}
 \tilde{\I}(v) \geq \min\brac{\I(\lambda_0) ,  \frac{1}{4R_0} , \frac{\exp(-\frac{\kappa}{2} R_0^2) }{4R_0}} \;\;\; \forall v \in [\lambda_0,1/2) ~.
\end{equation}

\subsection{Summary}

We conclude from (\ref{eq:small}) and (\ref{eq:large}) that:
\begin{equation} \label{eq:thm2-final-bound}
 \tilde{\I}(v) \geq \min\brac{c_{\delta_0} v \gamma(\log 1/v) , c_{\delta_0} \lambda_0 \gamma(\log 1/\lambda_0) , \frac{\exp(-\frac{\kappa}{2} R_0^2) }{4R_0} } \;\;\; \forall v \in [0,1/2] ~,
\end{equation}
where $\lambda_0 = \lambda_0(\delta_0,\alpha(r_0))$, $c_{\delta_0}$ and $R_0 = R_0(\lambda_0,\alpha)$ are given by (\ref{eq:lambda_0}), (\ref{eq:c_0}) and (\ref{eq:R_0}), and $\gamma(x) = x / \alpha^{-1}(x)$. Note that when $\kappa > 0$, the dependence of this bound on $\alpha$ may be simplified, by using that $\delta_0 \kappa \alpha^{-1}(\log 1/\lambda_0)^2 \leq \log 1/\lambda_0$. The dependence of the parameters asserted in the statement of Theorem \ref{thm:2} follows.
 
\section{Approximation Argument} \label{sec:aa}

In this section, we extend the proof of Theorem \ref{thm:2} (and hence \ref{thm:1}) to our general $\kappa$-semi-convexity assumptions, removing all smoothness assumptions, which were absolutely crucial for using the existence and regularity theory of isoperimetric minimizers, and also required for using the generalized Heintze--Karcher theorem. We emphasize that it is not true in general that if $\set{\mu_m}$ tends to $\mu$ in total variation on a common metric space $(\Omega,d)$, then necessarily $\I_{(\Omega,d,\mu_m)}$ tends to $\I_{(\Omega,d,\mu)}$, even pointwise (see e.g. \cite{EMilman-RoleOfConvexity,EMilmanRoleOfConvexityInFunctionalInqs}).
Fortunately, a one-sided relation always exists. Recall that $\set{\mu_m}$ is said to converge to $\mu$ in total-variation if:
\[
d_{TV}(\mu_m,\mu) := \sup_{A \subset \Omega} \abs{\mu_m(A) - \mu(A)} \rightarrow_{m \rightarrow \infty} 0 ~.
\]
The following lemma follows directly from the proof of \cite[Lemma 6.6]{EMilman-RoleOfConvexity}:

\begin{lem} \label{lem:aa1}
Assume that $\set{\mu_m}$ is a sequence of Borel probability measures on a common metric space $(\Omega,d)$, which tends to $\mu$ in total-variation. Then:
\begin{equation} \label{eq:I-one-sided}
\liminf_{u \rightarrow v} \I_{(\Omega,d,\mu)}(u) \geq \lim_{\eps \rightarrow 0} \limsup_{m \rightarrow \infty} \inf_{|u-v| < \eps} \I_{(\Omega,d,\mu_m)}(u) \;\;\; \forall v \in (0,1) ~.
\end{equation}
\end{lem}

Assume now that $(\Omega,d,\mu)$ satisfies our $\kappa$-semi-convexity assumptions. This means that $\mu$ is an absolutely continuous Borel probability measure on a complete smooth oriented Riemannian manifold $\Omega = (M,g)$, with $d$ the induced geodesic distance, and that there exists a sequence $\set{\mu_m}$ of Borel probability measures on $(\Omega,d)$, which converges in total-variation to $\mu$, so that $(\Omega,d,\mu_m)$ satisfies our smooth $\kappa$-semi-convexity assumptions for every $m$.

Under these semi-convexity assumptions, by following the proofs of \cite[Lemma 6.9]{EMilman-RoleOfConvexity} and \cite[Lemma 7.1]{EMilmanRoleOfConvexityInFunctionalInqs}, it is in fact possible to show (perhaps after passing to a subsequence) that $\I_{(\Omega,d,\mu_m)}$ are locally uniformly continuous on $(0,1)$, so that the limits in the right-hand side of (\ref{eq:I-one-sided}) may be exchanged. We will not develop this here, and instead bypass this using an argument which only depends on the estimates obtained in the proof of Theorem \ref{thm:2}.

Recall that $\K = \K_{(\Omega,d,\mu)}: \Real_+ \rightarrow [\log 2, \infty]$ denotes the exponential concentration profile of $(\Omega,d,\mu)$, so that $\K(r)$ is the best possible constant in:
\[
\mu(A) \geq 1/2 \;\;\; \Rightarrow \;\;\; 1 - \mu(A^d_r) \leq \exp(-\K(r)) ~.
\]
Clearly, $\K$ is non-decreasing, and since $\mu$ is assumed absolutely continuous, we have $\K(0) = \log(2)$. It should also be true under our semi-convexity assumptions that $\K$ is in fact increasing and continuous from the right, but we will not insist on this here (a rigorous proof seems to be rather involved). Let us therefore fix our convention for the inverse of a non-decreasing function $\alpha:\Real_+ \rightarrow \Real$:
\begin{equation} \label{eq:inverse-conv}
\alpha^{-1}(s) := \sup \set{ r \geq 0 ; \alpha(r) \leq s } ~.
\end{equation}
Before proceeding, let us summarize some useful properties of this convention:
\begin{lem} \label{lem:inverse-prop}
For a non-decreasing function $\alpha:\Real_+ \rightarrow \Real$, we have:
\begin{enumerate}
\item $\alpha \circ \alpha^{-1} \geq id$.
\item $\alpha^{-1} \circ \alpha \geq id$.
\item $\alpha(\alpha^{-1}(s) - \delta) \leq s$ for any $0<\delta<\alpha^{-1}(s)$ and $s \geq \alpha(0)$.
\item $\alpha^{-1}(\alpha(r) -\delta) \leq r$ for any $0<\delta<\alpha(r)-\alpha(0)$ and $r \geq 0$.
\item $\alpha^{-1}$ is continuous from the right.
\end{enumerate}
\end{lem}

\begin{lem} \label{lem:alpha-approx}
The general case of Theorems \ref{thm:1} and \ref{thm:2} follows from the case that $\alpha$ is strictly increasing and continuous.
\end{lem}
\begin{proof}
Given a general non-decreasing $\alpha$ tending to infinity and satisfying the growth condition (\ref{eq:alpha-cond}), and any $\delta_0' \in (1/2,\delta_0)$, it is elementary to approximate $\alpha$ by continuous, strictly increasing functions $\set{\alpha_k}$, so that $\alpha \geq \alpha_k$, $\alpha_k(r) \geq \delta_0' \kappa r^2$ for any $r \geq r_0 $, and so that $\alpha_k^{-1}$ pointwise converge to $\alpha^{-1}$ (the latter would not be possible had we used a convention different from (\ref{eq:inverse-conv})). Indeed, simply consider $\alpha^{-1}$, which is continuous from the right by Lemma \ref{lem:inverse-prop} and non-decreasing, and hence upper semi-continuous, and express it as the infimum of a sequence of continuous strictly increasing functions $\set{\beta_k}$; the functions $\alpha_k := \beta_k^{-1}$ will have the desired properties. Since the lower bound on $\tilde{\I}(v)$ in the conclusion of Theorems \ref{thm:1} and \ref{thm:2} depends continuously on $\delta_0$, $\alpha^{-1}(\log 1/v)$ and an upper bound on $\alpha(r_0)$, the reduction follows by taking an appropriate limit.
\end{proof}

We will denote $\K_m = \K_{(\Omega,d,\mu_m)}$, $\I_m = \I_{(\Omega,d,\mu_m)}$ and $\I = \I_{(\Omega,d,\mu)}$ for short, and use the notation $f(r+) := \lim_{r' \rightarrow r+} f(r')$.

\begin{lem} \label{lem:basic}
Under the assumptions of Lemma \ref{lem:aa1}:
\begin{eqnarray}
\nonumber & & \forall v \in (0,1/2) \;\; \forall 0<\eps<v/2 \;\; \exists m_\eps \;\; \forall m \geq m_\eps \\
\label{eq:r_eps} & & \K_m^{-1}(\log 1/v) \leq \K^{-1}\brac{\log \frac{1}{1/2 - \eps}} + \K^{-1}\brac{\log \frac{1}{{v-2\eps}}} ~.
\end{eqnarray}
\end{lem}
\begin{proof}
Given $v \in (0,1)$, let $0<\eps<v/2$.
Denoting by $r_\eps$ the right hand side of (\ref{eq:r_eps}),
it follows from Lemma \ref{lem:inverse-prop} and the definition of $\K$ (applied twice) that:
\begin{equation} \label{eq:main-lem1}
\mu(B) \geq 1/2 - \eps \; \Rightarrow \; \mu\brac{\overline{B_{\K^{-1}\brac{\log \frac{1}{1/2 - \eps}}}}} \geq 1/2 \; \Rightarrow \; 1 - \mu(B_{r_{\eps}}) \leq v - 2\eps ~.
\end{equation}
Since $\set{\mu_m}$ tends to $\mu$ in total-variation, there exists $m_\eps$ so that $d_{TV}(\mu_m,\mu) < \eps$ for all $m \geq m_\eps$.
Let $A_m \subset \Omega$ denote a Borel set so that $\mu_m(A_m) \geq 1/2$ and $\exp(-\K_m(r_{\eps})) - \eps \leq 1 - \mu_m((A_m)_{r_{\eps}})$.
Since $\mu(A_m) \geq 1/2 - \eps$ for $m \geq m_\eps$, it follows from (\ref{eq:main-lem1}) that:
\[
\exp(-\K_m(r_{\eps})) - \eps \leq 1 - \mu_m((A_m)_{r_{\eps}}) < 1 - \mu((A_m)_{r_{\eps}}) + \eps \leq v - 2\eps + \eps ~.
\]
Therefore, for such $m$, $\K_m(r_{\eps}) > \log 1/v$, and Lemma \ref{lem:inverse-prop} implies that $r_{\eps} \geq \K^{-1}_m(\log 1/v)$, which is the asserted claim.
\end{proof}

\begin{lem} \label{lem:uniform-growth}
Under the assumptions of Lemma \ref{lem:aa1}, if $\K$ satisfies the growth condition:
\begin{equation} \label{eq:K-cond}
\exists \delta_0 > 1 \;\;\; \exists r_0  \;\;\; \forall r \geq r_0 \;\;\; \K(r) \geq \delta_0 \kappa r^2 ~,
\end{equation}
for some $\kappa \geq 0$, then:
\begin{eqnarray*}
& & \forall 1<\delta_1<\delta_0 \;\;\; \exists \lambda = \lambda(\K,\kappa,\delta_1) \;\;\; 0 < \lambda \leq \min(\exp(-\K(r_0)),1/2) \\
& & \forall v \in (0,\lambda] \;\; \exists m_{v,\delta_1}\;\;\; \forall m \geq m_{v,\delta_1} \;\;\; \log 1/v \geq \delta_1 \kappa \K_m^{-1}(\log 1/v)^2 ~.
\end{eqnarray*}
\end{lem}
\begin{proof}
The growth condition and Lemma \ref{lem:inverse-prop} imply that $s \geq \delta_0 \kappa \K^{-1}(s)^2$ for all $\K^{-1}(s) > r_0$. Using Lemma \ref{lem:inverse-prop} again, this means that the latter holds with $s = \log 1/(v - 2\eps)$ for any $0 < 2\eps < v < \min(\exp(-\K(r_0+))+2\eps,1/2)$. Lemma \ref{lem:basic} then implies that for any $m \geq m_\eps$:
\begin{eqnarray*}
\frac{\sqrt{\log 1/v}}{\sqrt{\kappa} \K_m^{-1} (\log 1/v)} &\geq& \frac{\sqrt{\log 1/v}}{\sqrt{\kappa} \brac{\K^{-1}\brac{\log \frac{1}{1/2-\eps}} + \K^{-1}\brac{\log \frac{1}{v-2\eps}}}} \\
&\geq&
\frac{\sqrt{\log 1/v}}{\sqrt{\kappa} \K^{-1}\brac{\log \frac{1}{1/2-\eps}} + \frac{1}{\sqrt{\delta_0}}\sqrt{\log \frac{1}{v-2\eps}}} ~.
\end{eqnarray*}
At this point we do not know that $\lim_{s \rightarrow \log 2 +} \K^{-1}(s) = 0$, but in any case the first term in the denominator above is bounded.  Therefore, for any $1<\delta_1<\delta_0$, we may choose $v > 0$ and $\eps = \eps(v) > 0$ small enough, so that the above expression is greater than $\sqrt{\delta_1}$. Setting $m_{v,\delta_1} = m_{\eps(v)}$, the asserted claim follows.
\end{proof}

This means that a uniform growth condition is satisfied by $\K_m$ for sets of small measure.
Analyzing the proof of Theorem \ref{thm:2} for small sets, we conclude the following:

\begin{cor}
Under our $\kappa$-convexity assumptions and the growth assumption (\ref{eq:K-cond}):
\begin{eqnarray*}
& & \exists \delta_1 > 1 \;\; \exists c_{\delta_1} > 0 \;\; \exists 0 < \lambda_1 < \min(\exp(-\K(r_0)),1/2) \;\;\; \exists m_1 \;\;\; \forall m \geq m_1 \\
& & \tilde{\I}_m(\lambda_1) \geq c_{\delta_1} \lambda_1 \frac{\log 1/\lambda_1}{\K_m^{-1}(\log 1/\lambda_1)} \geq c_{\delta_1} \lambda_1 \sqrt{\delta_1 \kappa} \sqrt{\log 1/\lambda_1}  ~.
\end{eqnarray*}
\end{cor}
\begin{proof}
Let $1 < \delta_1 < \delta_0$, and set $c_{\delta_1}>0$ according to (\ref{eq:c_0}). According to the proof of Theorem \ref{thm:2} (cases 2 and 3 of Subsection \ref{subsec:small-sets}), which remains valid for a general non-decreasing $\alpha$ by Lemma \ref{lem:inverse-prop}, in order to obtain the claimed lower bound on $\tilde{\I}_m(\lambda_1)$, we need the growth condition $\log 1/v \geq \delta_1 \kappa \K_m^{-1}(\log 1/v)^2$ to hold for $v = \lambda_1, \lambda_1/2$. Lemma \ref{lem:uniform-growth} ensures that this happens for all $m \geq m_1$ with $m_1 = \max(m_{\lambda_1,\delta_1},m_{\lambda_1/2,\delta_1})$. The second inequality above follows from the same uniform upper bound on $\K_m^{-1}(\log 1/\lambda_1)$.
\end{proof}

The proof of Theorem \ref{thm:2} for large sets (Subsection \ref{subsec:large-sets}), where the estimates only depend on $\K_m^{-1}(\log 1/\lambda_1)$ and $\K_m^{-1}(\log 4)$, which are uniformly bounded from above by Lemma \ref{lem:basic}, then gives:

\begin{prop} \label{prop:I-is-positive}
Under our $\kappa$-convexity assumptions and the growth assumption (\ref{eq:K-cond}):
\begin{eqnarray*}
& & \exists c_{\K,\kappa} > 0 \; \exists 0<\lambda_1 < \min(\exp(-\K(r_0)),1/2) \; \exists m_1 \\
& & \forall m \geq m_1 \;\; \forall v \in [\lambda_1,1/2] \;\;\; \tilde{\I}_m(v) \geq c_{\K,\kappa} ~.
\end{eqnarray*}
Consequently, Lemma \ref{lem:aa1} implies:
\[
\forall v \in (\lambda_1,1/2] \;\;\; \tilde{\I}(v) \geq c_{\K,\kappa} > 0 ~.
\]
\end{prop}

Recalling that an isoperimetric inequality always implies a concentration inequality, as given by (\ref{eq:i-c}), we easily obtain:

\begin{cor} \label{cor:cont-at-log2}
Under our $\kappa$-semi-convexity assumptions and the growth assumption (\ref{eq:K-cond}), $\K^{-1}(s)$ is continuous at $s=\log 2$, so $\lim_{s \rightarrow \log 2+} \K^{-1}(s) = 0$.
\end{cor}
\begin{rem}
Had we a-priori known that $\K(r) > \log 2$ for all $r>0$, this conclusion would immediately follow by showing that $\K(r)$ is continuous at $r=0$, which is elementary. Our more complicated argument verifies this in a strong sense, by showing that the derivative of $\K(r)$ at $r=0$ must be strictly positive.
\end{rem}

We can now claim:

\begin{prop}
Theorem \ref{thm:2} holds without any smoothness assumptions.
\end{prop}
\begin{proof}
An inspection of the proof of Theorem \ref{thm:2}, valid under our additional smoothness assumptions, gives for the space $(\Omega,d,\mu_m)$:
\begin{equation} \label{eq:known}
 \tilde{\I}_m(v) \geq J_{m,\lambda}(v) \;\;\; \forall v \in [0,1/2] \;\;\; \forall \lambda \in (0,1/4] ~,
\end{equation}
where:
\[
J_{m,\lambda}(v) := \begin{cases} c_{\delta_m(v)} \frac{v \log 1/v}{K_m^{-1}(\log 1/v)} & v \in [0,\lambda] \\
 \min\brac{c_{\delta_m(\lambda)} \frac{\lambda \log 1/\lambda}{K_m^{-1}(\log 1/\lambda)} , \frac{\exp(-\frac{\kappa}{2} R_{m,\lambda}^2) }{4R_{m,\lambda}} } & v \in [\lambda,1/2] \end{cases} ~,
\]
\[
\delta_m(u) := \min\brac{ \frac{\log 1/u}{\kappa \K_m^{-1}(\log 1/u)^2} , \frac{\log 2/u}{\kappa \K_m^{-1}(\log 2/u)^2}} ~,
\]
\[
c_{\delta_m(u)} := \frac{e}{16e+2} \brac{1 - \frac{1}{2\delta_m(u)}} ~,
\]
and:
\[
R_{m,\lambda} := \K_m^{-1}(\log 1/\lambda) + \K_m^{-1}(\log 4) ~.
\]

Note that (\ref{eq:known}) is only meaningful when $c_{\delta_m(v)} > 0$ for all $v \in (0,\lambda]$.
Since we assume that $\K \geq \alpha$, we will set $\K_0 := \alpha$ in the notations above, and reserve the subscript $m=0$ to refer to this case.
By Lemma \ref{lem:inverse-prop}, $\alpha^{-1}$ is always continuous from the right, and since $\K^{-1}$ is continuous at $\log 2$ by Corollary \ref{cor:cont-at-log2}, it follows by Lemma \ref{lem:basic} that:
\[\lim_{\eps \rightarrow 0} \liminf_{m \rightarrow \infty} \inf_{|u-v|<\eps} \frac{\delta_m(u)}{\delta_0(v)} \geq 1 ~,~  \lim_{\eps \rightarrow 0} \liminf_{m \rightarrow \infty} \inf_{|u-v|<\eps} \frac{\K_0^{-1}(v)}{\K_m^{-1}(u)} \geq 1 ~.
\]Consequently:
\[\lim_{\eps \rightarrow 0} \liminf_{m \rightarrow \infty} \inf_{|u-v|<\eps} \frac{R_{0,\lambda}}{R_{m,\lambda}} \geq 1 ~,
~ \lim_{\eps \rightarrow 0} \liminf_{m \rightarrow \infty} \inf_{|u-v|<\eps} \frac{c_{\delta_m(u)}}{c_{\delta_0(v)}} \geq 1 ~.
\] Putting everything together, this implies that:
\[
\lim_{\eps \rightarrow 0} \liminf_{m \rightarrow \infty} \inf_{|u-v|<\eps} J_{m,\lambda}(u) \geq J_{0,\lambda}(v) ~.
\]
Lemma \ref{eq:I-one-sided} therefore implies:
\begin{eqnarray*}
\I(v) & \geq & \lim_{\eps \rightarrow 0} \limsup_{m \rightarrow \infty} \inf_{|u-v|<\eps} \I_m(u) \\
& \geq & \lim_{\eps \rightarrow 0} \liminf_{m \rightarrow \infty} \inf_{|u-v|<\eps} J_{m,\lambda}(u) \geq
J_{0,\lambda}(v) ~.
\end{eqnarray*}

It remains to choose $\lambda =\min(\exp(-\alpha(r_0)),\exp(-\frac{\log 16}{1 - 1/(2\delta_0)}),1/4)$
as in the proof of Theorem \ref{thm:2}, which ensures that $\delta_0(v) \geq \delta_0 > 1/2$ for all $v \in (0,\lambda]$ by our growth condition (\ref{eq:alpha-cond}), and hence $c_{\delta_0(v)} > 0$ in that range. We thereby recover the result (\ref{eq:thm2-final-bound}) of Theorem \ref{thm:2}, without any smoothness assumptions. The proof is complete.
\end{proof}

\section{Applications} \label{sec:applications}

In this section, we provide a couple of applications of Theorems \ref{thm:1} and \ref{thm:2}, which complete the picture we have chosen to depict in this work. We defer the description of many other applications to \cite{EMilmanGeometricApproachPartII,EMilmanGeometricApproachPartIII}.

\subsection{Any Concentration Implies Linear Isoperimetry}

As an application of Theorem \ref{thm:1} and Proposition \ref{prop:weak-concave}, we can immediately deduce the Main Theorem of our previous work \cite{EMilman-RoleOfConvexity}. Our original approach involved many ingredients, such as semi-group gradient estimates following Bakry--Ledoux \cite{BakryLedoux} and Ledoux \cite{LedouxSpectralGapAndGeometry}, a Paley--Zygmund type argument, and Theorem \ref{thm:concave} on the concavity of the isoperimetric profile $\I$, which is based on an analysis of the second variation of area. The following argument avoids all of these ingredients, simplifying the machinery underlying the proof. Moreover, we obtain the following slightly stronger version:

\begin{thm} \label{thm:main-better}
Assume that:
\[
\exists \lambda_0 \in (0,1/2) \;\; \exists r_0 > 0 \;\; \text{such that} \;\; \mu(A) \geq 1/2 \Rightarrow 1-\mu(A^d_{r_0}) \leq \lambda_0 ~.
\]
Then under our convexity assumptions on $(\Omega,d,\mu)$:
\[
 \frac{\tilde{\I}(v)}{v} \geq \frac{k_{\lambda_0}}{r_0} \;\;\; \forall v \in [0,1/2] ~,
\]
for some constant $k_{\lambda_0}> 0$ depending solely on $\lambda_0$.
\end{thm}

This result for a small enough (non explicitly calculated) value of $\lambda_0>0$ follows from \cite[Theorem 1.5 and Corollary 4.3]{EMilman-RoleOfConvexity}. We now see that \emph{any} non-trivial uniform concentration implies a linear isoperimetric inequality under our convexity assumptions.

\begin{proof}
We will provide a proof using our additional smoothness assumptions; the general case follows by the arguments of Section \ref{sec:aa}.
Note that our concentration assumption is precisely that $\K(r_0) \geq \log 1/\lambda_0$, implying by Lemma \ref{lem:inverse-prop} that $\K^{-1}(\log 1 / (\lambda_0 +\eps)) \leq r_0$, for any small $\eps > 0$. According to Corollary \ref{cor:weak-concave}, $\inf_{v \in [0,1/2]} \tilde{\I}(v) / v = 2 \I(1/2)$, so it is enough to obtain a lower bound on $\I(1/2)$. But this readily follows from Theorem \ref{thm:1}, since according to (\ref{eq:thm1-conc-strong}) or (\ref{eq:thm1-conc-weak}), setting $\lambda_\eps = \lambda_0 + \eps$, we have for any $\eps \in (0,1/2 - \lambda_0)$:
\[
 \frac{\tilde{\I}(v)}{v} \geq 2 \I(1/2) \geq 2 c_{\lambda_\eps} \frac{\lambda_\eps}{1-\lambda_\eps} \frac{\lambda_\eps \log 1/\lambda_\eps}{\K^{-1}(\log 1 / \lambda_\eps)} \geq 2 c_{\lambda_\eps} \frac{\lambda_\eps}{1-\lambda_\eps} \frac{\lambda_\eps \log 1/\lambda_\eps}{r_0} ~.
\]
Taking the limit as $\eps$ tends to $0$, we also obtain the formula for $k_{\lambda_0}$.
\end{proof}

\subsection{Generalized Bobkov's Isoperimetric Inequality}

Let us see how Theorems \ref{thm:1} and \ref{thm:2} easily imply a generalized version of Bobkov's isoperimetric inequality (Corollary \ref{cor:gen-Bobkov}).

\begin{proof}[Proof of Corollary \ref{cor:gen-Bobkov}]
The crucial yet elementary observation is that if $\beta$ satisfies (\ref{eq:beta}) then:
\begin{equation} \label{eq:Bobkov}
\K(r + \beta^{-1}(\log 2)) \geq \beta(r) \;\;\; \forall r \geq 0 ~.
\end{equation}
Indeed, set $R := \beta^{-1}(\log 2)$ and denote $B_r := \set{x \in \Omega ; d(x,x_0) < r}$, so that $\mu(B_R) \geq 1 - \exp(-\beta(R)) \geq 1/2$. Given any Borel subset $A \subset \Omega$ with $\mu(A) \geq 1/2$, we must either have $A \cap B_R \neq \emptyset$ or $\partial A \cap \partial B_R \neq \emptyset$ (the latter possibility follows since otherwise we would conclude by compactness that $\mu(B_{R+\eps}) = 1/2$ for some $\eps>0$, in contradiction to the definition of $R$ and our convention (\ref{eq:inverse-conv})). Therefore $x_0 \in \overline{A^d_R}$, hence $B_r \subset A^d_{r+R}$, and (\ref{eq:Bobkov}) immediately follows. Defining:
\[
\alpha(r) = \begin{cases} 0 & r \in [0,R] \\ \beta(r - R) & r \in [R,\infty) \end{cases} ~,
\]
we see that $\K \geq \alpha$ and that $\alpha$ satisfies a growth condition similar to that of $\beta$:
\[
\exists \delta_0' > 1/2 \;\;\; \exists r_0' \geq 0 \;\;\; \forall r \geq r_0' \;\;\; \alpha(r) \geq \delta_0' \kappa r^2 ~,
\]
with say $\delta_0' = \frac{1}{4} + \frac{\delta_0}{2}$ and $r_0' = \max(r_0,b_{\delta_0} R) + R$, where $b_{\delta_0}$ depends solely on $\delta_0$.
Applying Theorems \ref{thm:1} or \ref{thm:2}, and noting that the definition of $\alpha$ implies:
\[
\alpha^{-1}(\log 1/v) = \beta^{-1}(\log 1/v) + \beta^{-1}(\log 2) \leq 2\beta^{-1}(\log 1/v) \;\;\; \forall v \in (0,1/2] ~,
\]
the conclusion of the corollary immediately follows.
\end{proof}

\section{Concluding Remarks} \label{sec:conclude}

\subsection{Global Vs. Local Convexity}

It is not difficult to verify that all of our results remain valid when our notion of smooth $\kappa$-semi-convexity is replaced by a slightly more general one: the measure $\mu$ is allowed to be supported on a \emph{geodesically convex} set $\Omega \subset (M,g)$ with $C^2$ boundary, so that if $d\mu = \exp(-\psi) dvol_M|_{\Omega}$, then $\psi \in C^2(\overline{\Omega})$ and $Ric_g + Hess_g \psi \geq -\kappa g$ on $\Omega$. This follows from the known results on orthogonality of the isoperimetric minimizers to the boundary of $\Omega$ in this setting (see Gr\"{u}ter \cite{Gruter} and also \cite[Appendix A]{EMilman-RoleOfConvexity} for a generalization to the case of a manifold-with-density). The geodesic convexity of $\Omega$ then implies that $(A^d_t \cap \Omega) \setminus A \subset \Phi(\partial_r A \times [0,t)) \cup \partial_s A$ for any isoperimetric minimizer $A$, where $\partial_r A$ and $\partial_s A$ denote the regular and singular parts of $\partial A$ respectively, and
$\Phi(x,t) = \exp_x(t \nu(x))$ is the normal map in the direction of the outer normal field $\nu$.
Consequently, the various versions of the Heintze--Karcher theorem and Theorem \ref{thm:tool} remain valid. The orthogonality to the boundary also ensures that the formula for the \emph{first} variation of area remains unaltered and therefore so does Proposition \ref{prop:H-is-derivative}  (see e.g. Sternberg--Zumbrun \cite{SternbergZumbrun}, Bayle--Rosales \cite{BayleRosales}); on the other hand, the \emph{second} variation in the presence of a boundary becomes even more complicated. A tedious computation invoking the second variation demonstrates (see Section \ref{sec:1vs2} and the references therein) that Theorem \ref{thm:concave} on the concavity of the isoperimetric profile also remains valid in this setting.

The assumption of geodesic convexity is a \emph{global} convexity assumption. This is one drawback of using the Heintze--Karcher theorem compared to second variation of area methods. The latter ones allow one to deduce a second order differential inequality on the isoperimetric profile, just from a weaker \emph{local} convexity assumption on the boundary of $\Omega$ - the requirement that the second fundamental form on $\partial \Omega$ be non-negative (see \cite{SternbergZumbrun}, \cite{BayleRosales} and \cite[Appendix A]{EMilman-RoleOfConvexity}).

\subsection{Constant Density Case}

When $\mu$ has constant density and our smooth convexity assumptions are satisfied, one may replace Theorem \ref{thm:gen-HK} by the classical Heintze--Karcher theorem in the derivation of our results (as in Subsection \ref{subsec:combining}, it still applies to non-entirely smooth isoperimetric minimizers). This would lead to some slightly more refined versions of Theorems \ref{thm:1} and \ref{thm:2}, recovering the correct behaviour of $\tilde{I}(v)$ as $v$ tends to $0$, namely $v^{\frac{n-1}{n}}$, where $n$ is the dimension of the manifold $M$. We do not pursue this point here.

Let us only mention that the same argument as in Proposition \ref{prop:weak-concave} would imply that $\I^{\frac{n}{n-1}}(v)/v$ is non-increasing on $(0,1)$. In the range $(0,1/2]$, this was previously shown by Gallot \cite[Corollary 6.6]{GallotIsoperimetricInqs} by an argument similar to ours, which in addition used a comparison with a model space. This is also in agreement with the stronger result known in this case \cite{Kuwert,BayleRosales}, stating that in fact $\I^{\frac{n}{n-1}}$ is concave. In addition, when $\mu$ has constant density, one may show that a similar lower bound to the one in the statement of Theorem \ref{thm:main-better} holds for the ratio $\tilde{\I}(v) / v^{\frac{n-1}{n}}$.

\subsection{Positive Curvature}

It is also possible to obtain better estimates in the case $\kappa = -\rho < 0$, i.e. when:
\begin{equation} \label{eq:CD-pos}
 Ric_g + Hess_g \psi \geq \rho g ~,~ \rho > 0 ~.
\end{equation}
As described in the Introduction, this case has been extensively studied by many authors, and many of the tools described in Section \ref{sec:prelim} have been used primarily to handle this case. Gromov \cite{GromovGeneralizationOfLevy} treated (compact) manifolds with uniform density and $Ric_g \geq \rho g$, obtaining a sharp isoperimetric inequality generalizing that of P. L\'evy for the sphere. By taking into account the diameter of the manifold, B\'erard--Besson--Gallot \cite{BBGImprovingGromov} were able to further improve on Gromov's estimate (and also treat the case $\rho \leq 0$, deriving a dimension dependent bound). This is roughly in the spirit of our idea - using information from concentration inequalities (in this case, that $\K(r) = +\infty$ for all $r$ greater than the diameter) to obtain or improve isoperimetric inequalities.

Things become even more interesting when considering manifolds-with-density satisfying $(\ref{eq:CD-pos})$. Bakry and \'Emery \cite{BakryEmery} were the first to realize that the condition (\ref{eq:CD-pos}) correctly captures the interplay between the geometry and the measure, and derived a sharp log-Sobolev inequality in this case by using a corresponding diffusion semi-group. This was substantially strengthened by Bakry and Ledoux \cite{BakryLedoux}, who obtained from (\ref{eq:CD-pos}) a sharp Gaussian isoperimetric inequality by using the semi-group method together with a functional form of the Gaussian isoperimetric inequality due to Bobkov \cite{BobkovFunctionalFormOfGaussianIsopInq}. In \cite{BobkovLocalizedProofOfGaussianIso}, Bobkov was able to reproduce this result in the Euclidean setting by using a geometric localization technique, which does not seem to handle more general manifolds. An elegant geometric proof of this inequality in the general manifold-with-density setting, which in addition characterizes the equality cases, was recently obtained by Morgan \cite{MorganManifoldsWithDensity},\cite[Chapter 18]{MorganBook4Ed}.

In many respects this is a satisfactory stopping point for the study of isoperimetric inequalities on manifolds-with-density satisfying (\ref{eq:CD-pos}). But one may still wonder how to combine (\ref{eq:CD-pos}) with \emph{additional} information from concentration inequalities. This may now be accomplished by repeating the relevant parts of the proofs of Theorems \ref{thm:1} and \ref{thm:2}. Alternatively, as noted to us by one of the referees, if one does not care about obtaining the best possible numeric constants, one may simply superimpose the additional concentration information with the concentration resulting from the Bakry--\'Emery or Bakry--Ledoux results above, and employ Theorem \ref{thm:1}, since all the information given by (\ref{eq:CD-pos}) will already be encoded in the resulting concentration inequality.

\subsection{Dimension Dependence}

Before concluding, we would like to explain our remark from the Introduction, regarding the dimension dependence of all the previous results in the spirit of Corollary \ref{cor:gen-Bobkov}. These results were derived under the smooth $\kappa$-semi-convexity assumptions, typically showing that the weaker integrability condition:
\[
 Q := \int_\Omega \exp(\beta(d(x,x_0))) d\mu < \infty
\]
for some function $\beta$ increasing to infinity, implies a corresponding isoperimetric inequality, with bounds depending on $Q$. Under the natural normalization that $\snorm{\frac{d\mu}{dvol_M}}^{1/n}_{L_\infty} \leq C$, it is possible to show that $Q$ must depend on the dimension $n$ of the underlying manifold $M$, rendering these results dimension dependent. Indeed, by the Markov--Chebyshev inequality:
\[
 Q \geq \mu\set{x \in \Omega ; d(x,x_0) \geq R} \exp(\beta(R)) ~,
\]
so one just needs to bound the measure of geodesic balls. But even in the case that $\mu$ has constant density (as above) with respect to $vol_M$, Bishop's volume-comparison theorem \cite{GHLBook} implies under our $\kappa$-semi-convexity assumptions that:
\[
 \mu\set{x \in \Omega ; d(x,x_0) \leq R} \leq C^n \textrm{Vol}(S^{n-1}) \int_0^R \brac{\frac{ \sinh(\sqrt{\kappa} r)}{\sqrt{\kappa}}}^{n-1} dr ~,
\]
where the integrand should be interpreted as $r^{n-1}$ if $\kappa = 0$, and $\textrm{Vol}(S^{n-1})$ denotes the Lebesgue measure of the $n-1$ dimensional Euclidean unit sphere, which is known to be of the order of $n^{-\frac{n-2}{2}}$. Consequently, $R$ must be at least of the order of $\frac{\log \sqrt{n \kappa}}{\sqrt{\kappa}}$ for $\kappa \gg 1/n$ and $\sqrt{n}$ otherwise, to ensure that the measure of the ball of radius $R$ is $1/2$, yielding a bad dimension dependent lower bound for $Q$.

\def\cprime{$'$}

\end{document}